\documentclass{lmcs}
\pdfoutput=1
\usepackage[utf8]{inputenc}

\usepackage{lastpage}
\lmcsdoi{22}{4}{8}
\lmcsheading{}{\pageref{LastPage}}{}{}%
{Jan.~23,~2026}{Oct.~27,~2026}{}


\usepackage{amssymb}
\usepackage{tikz-cd}
\usepackage{xurl}

\newcommand{\CC}{\mathbb{C}}
\newcommand{\U}{\mathcal{U}}
\newcommand{\Udot}{\dot{\mathcal{U}}}

\newcommand{\pbcorner}{\ar[dr, phantom, very near start, "\lrcorner"]}

\mathchardef\mhyphen="2D

\begin{document}

\title{A categorical formulation of Kraus' paradox}

\author{Andrew W Swan\lmcsorcid{0000-0002-7190-4870}}
\address{Faculty of Mathematics and Physics, University of Ljubljana,
  Jadranska 21, 1000 Ljubljana, Slovenia}
\email{wakelin.swan@gmail.com}
\thanks{This material is based upon work supported by the Air Force Office
  of Scientific Research under award number FA9550-21-1-0024.}

%% the abstract has to PRECEDE \maketitle
\begin{abstract}
  We give a categorical formulation of Kraus' ``magic trick'' for recovering information from truncated types. Rather than type theory, we work in Van den Berg--Moerdijk path categories with a univalent universe, and rather than propositional truncation we work with arbitrary cofibrations, which includes truncation as a special case. We show, using Kraus' argument, that any cofibration with homogeneous domain is a monomorphism. We give some simple concrete examples in groupoids to illustrate the interaction between homogeneous types, cofibrations and univalent fibrations.
\end{abstract}

\maketitle

\section{Introduction}
\label{sec:introduction}

Propositional truncation is one of the most basic and most important examples of higher induction types \cite[Section 3.7]{hottbook}. An \emph{hProposition} is a type \(A\) such that any two elements of \(A\) are equal. We then define the propositional truncation of a type \(A\) to be a type \(\| A \|\) together with a map \(|-| : A \to \| A \|\) which makes \(A\) into an hProposition in the ``minimal way,'' in the sense that it satisfies a suitable induction principle. In \cite[Section 8.4]{kectanonex} Kraus et al.\ gave a surprising proof that under certain conditions we can take a term of the form \(|a|\) for \(a : A\) and ``extract'' the term \(a\) back out. More formally, we have a family of dependent types \(z : \| A \| \vdash B(z)\) and a dependent function \(z : \|A \| \vdash b(z) : B(z)\) such that for \(a : A\), \(B(|a|)\) is definitionally equal to \(A\) and \(b(|a|)\) is definitionally equal to \(a\). This is surprising, since by definition any two elements of \(\| A \|\) are equal, so we wouldn't expect to be able to distinguish different elements of \(\|A\|\) to the extent that we can even find which element of \(A\) they came from.

In both the original presentation of \cite{kectanonex} and in our summary above, Kraus' paradox is described in purely type theoretic terms, using terminology such as ``terms,'' ``induction,'' ``definitional equality,'' and ``computation rules.'' This might mislead one into thinking that Kraus' paradox is an artefact of the technical aspects of type theory and has no deeper meaning or significance outside type theory. We remedy this with a purely categorical presentation of Kraus' argument. We work in the general setting of univalent Van den Berg--Moerdijk path categories \cite{vdbergmoerdijkpathcat}, and show that any cofibration with homogeneous domain is a monomorphism.

\section{Path categories}
\label{sec:brown-fibr-categ}

We work in the general setting of path categories, as defined below.
\begin{defi}[Van den Berg--Moerdijk \cite{vdbergmoerdijkpathcat}]
  A \emph{path category} is a category~\(\CC\) together with two classes \(\mathcal{F}\) and \(\mathcal{W}\) of morphisms of \(\CC\) whose elements we respectively refer to as \emph{fibrations} and \emph{weak equivalences}, satisfying the following conditions. Below and throughout the paper, we will refer to maps in the intersection of \(\mathcal{F}\) and \(\mathcal{W}\) as \emph{trivial fibrations}.
  \begin{enumerate}
  \item Fibrations are closed under composition.
  \item The pullback of a fibration along any map exists and is also a fibration.
  \item The pullback of a trivial fibration along any map is a trivial fibration.
  \item Weak equivalences satisfy 2-out-of-6: for all maps \(f : A \to B\), \(g : B \to C\) and \(h : C \to D\), if both \(g \circ f\) and \(h \circ g\) are weak equivalences, then so are \(f\), \(g\), \(h\) and \(h \circ g \circ f\).
  \item Every isomorphism is a trivial fibration.
  \item Every trivial fibration has a section.
  \item \(\CC\) has a terminal object, and for every object \(B\), the unique map \(B \to 1\) is a fibration.
  \item For any object \(B\), the diagonal map \(B \to B \times B\) factors as \(B \stackrel{r}{\to} P B \stackrel{p}{\to} B \times B\) where \(r\) is a weak equivalence and \(p\) is a fibration. We refer to such a factorisation as a \emph{path object} on \(B\).
  \end{enumerate}
\end{defi}

\noindent Given a path object \(B \stackrel{r}{\to} P B \stackrel{p}{\to} B \times B\), we will write \(p_{0}\) for \(\pi_{0} \circ p\) and \(p_{1}\) for \(\pi_{1} \circ p\). Note that we get the 3-for-2 property as a special case of the 2-out-of-6 property, i.e.\ for maps \(f : A \to B\) and \(g : B \to C\), if any two of the maps \(f\), \(g\) and \(g \circ f\) is a weak equivalence, then so is the third.

\begin{defi}
  \label{defn:homotopieswellbehaved}
  Let \(A, B \in \CC\), and let \(P B\) be a path object on \(B\). For maps \(f, g : A \to B\), a \emph{homotopy} from \(f\) to \(g\) is a map \(H : A \to P B\) making the following diagram commute.
  \[
    \begin{tikzcd}
      & & B \\
      A \rar["H"] \ar[urr, bend left, "f"] \ar[drr, bend right, swap, "g"] & PB \ar[ur, "p_{0}"] \ar[dr, swap, "p_{1}"] & \\
      & & B
    \end{tikzcd}
  \]
  We say \(f\) and \(g\) are \emph{homotopic} and write \(f \sim g\) if there is a homotopy from \(f\) to \(g\).
\end{defi}

We can think of the definition of path category as the minimal structure on a category necessary for the above notion of homotopy to be well behaved, in the sense that we have the theorem below.
\begin{thm}
  \label{thm:calcfrac}
  \hfill
  \begin{enumerate}
  \item The definition of homotopy is independent of the particular choice of path object on \(B\).
  \item \label{it:calcfrac:eqreln} Homotopy gives an equivalence relation on the set of maps \(A \to B\).
  \item \label{it:calcfrac:intr} If \(f \sim g\), then \(f \circ h \sim g \circ h\) and \(k \circ f \sim k \circ g\).
  \end{enumerate}
\end{thm}

\begin{proof}
  See \cite[Theorem 2.14]{vdbergmoerdijkpathcat}.
\end{proof}

\begin{exa}
  \label{ex:extensionalfibcat}
  Let \(\CC\) be any category with finite limits. Then we can make \(\CC\) a path category by taking all maps to be fibrations, and the weak equivalences to be precisely the class of isomorphisms. We will refer to path categories of this form as \emph{extensional}.
\end{exa}
\begin{prop}
  The following are equivalent:
  \begin{enumerate}
  \item \label{it:extn}\(\CC\) is extensional.
  \item \label{it:hmptytriv}Two maps are homotopic iff they are equal.
  \item \label{it:allweiso}Every weak equivalence is an isomorphism.
  \end{enumerate}
\end{prop}
\begin{proof}
  For \ref{it:extn} implies \ref{it:hmptytriv}, note that if every map is a fibration, then in particular for any object \(B\), the diagonal map \(B \to B \times B\) is a fibration. It follows that taking \(P B = B\) gives a path object on \(B\). However, this implies that for any object \(A\), if two maps \(f, g : A \to B\) are homotopic, then they are equal.

  For \ref{it:hmptytriv} implies \ref{it:allweiso}, we recall that any weak equivalence is a homotopy equivalence by \cite[Theorem 2.16]{vdbergmoerdijkpathcat}, and so by \ref{it:hmptytriv} any homotopy equivalence is an isomorphism.

  For \ref{it:allweiso} implies \ref{it:extn}, we recall that any map \(f : A \to B\) factors as a weak equivalence followed by a fibration by \cite[Proposition 2.3]{vdbergmoerdijkpathcat}. It follows that \(f\) is a fibration, since fibrations contain isomorphisms and are closed under composition.
\end{proof}

\noindent
Although the above example is simple, it is a useful one. When working with weak equivalences, it is often useful to think of them as maps that behave ``like isomorphisms.'' In order to make this idea precise it is important that we can in fact view isomorphisms as a degenerate special case of weak equivalences. By placing categories with finite limits within the same general framework as other examples of path categories, such as Kan complexes in simplicial and cubical sets, we can easily compare them.

The example is also a useful one in type theory for the same reason. Any locally cartesian closed category is a model of extensional type theory \cite{seelylcctt} and also, since it has all finite limits, an extensional path category. Many more recent examples of models of type theory coming from homotopy theory \cite{hofmannstreichergroupoid, voevodskykapulkinlumsdainess, awodeywarren, coquandcubicaltt} very naturally give rise to path categories, and so we can use path categories as a common framework in which to compare the two different kinds of model.

\begin{exa}
  \label{ex:modelcat}
  The category of fibrant objects in a model structure where every object is cofibrant is a path category \cite[Section 2.1]{vdbergmoerdijkpathcat}.
\end{exa}

\begin{exa}
  As a special case of Example~\ref{ex:modelcat}, the category of Kan complexes \cite{gabrielzisman} is a path category with fibrations given by Kan fibrations and weak equivalences by homotopy equivalences.
\end{exa}

\begin{exa}
  \label{ex:groupoids}
  The category of groupoids is a path category, taking fibrations to be isofibrations and weak equivalences to be equivalences. See e.g.\ \cite{joyaltierneystacks} for more details. This is also a special case of Example~\ref{ex:modelcat}.
\end{exa}

\begin{exa}[Gambino--Garner \cite{gambinogarner}]
  \label{ex:ttfibcat}
  We can construct examples of path categories using type theory. We take \(\CC\) to be the category of contexts, fibrations to be the closure of projections \((\Gamma, x : A) \to \Gamma\) under retracts, and weak equivalences to be homotopy equivalences.
\end{exa}

\begin{exa}[Avigad--Kapulkin--Lumsdaine {\cite[Theorem 3.2.5]{avigadkapulkinlumsdaine}}]
  \label{ex:ttmodelfibcat}
  For any contextual category the underlying category of contexts can be viewed as a path category. This is a generalisation of Example~\ref{ex:ttfibcat}, by considering the syntactic contextual category of type theory. Furthermore, this example also works when identity types satisfy the computation rule only propositionally \cite{vdbergpathcatpropid}.
\end{exa}

\noindent
In any path category we can identify the important class of maps of \emph{cofibrations}. We will define them via the general notion of lifting property below.

\begin{defi}
  Given two maps \(m : A \to B\) and \(f : C \to D\), we say \(m\) has the \emph{left lifting property} against \(f\) and \(f\) has the \emph{right lifting property} against \(m\) if for every commutative square as in the solid lines in the diagram below, there is a diagonal map, as in the dotted line below making two commutative triangles.
  \[
    \begin{tikzcd}
      A \ar[d, "m"] \rar & C \dar["f"] \\
      B \ar[r] \ar[ur, dotted] & D
    \end{tikzcd}
  \]
\end{defi}

\begin{defi}
  A \emph{cofibration} is a map with the left lifting property against all trivial fibrations.
\end{defi}

\noindent
In keeping with the spirit of path categories, we do not assume any ``existence conditions'' on the class of cofibrations, so in principle it could be the case that a path category has ``very few'' cofibrations. However, in any path category we can identify the class of cofibrations and talk about what properties elements of this class must satisfy.

Having said that, most of the particular examples of path categories that we consider here are the fibrant objects of a model structure, and as such every map can be factored as a cofibration followed by a trivial fibration.

\begin{exa}
  In extensional path categories any map is a cofibration, since any map has the left lifting property against isomorphisms. In fact the converse also holds: If every map is a cofibration, then every trivial fibration is an isomorphism, but every map factors as a section of a trivial fibration followed by a fibration \cite[Proposition 2.3]{vdbergmoerdijkpathcat} and so also every map is a fibration.
\end{exa}

\begin{exa}
  The cofibrations in groupoids are precisely the functors that are injective on objects.
\end{exa}

\begin{exa}[Lumsdaine \cite{lumsdainecofibtt}]
  We consider the path category structure on the syntactic category of type theory (Example~\ref{ex:ttfibcat}). Any point constructor of any higher inductive type is a cofibration in the category of contexts, as long as the induction principle satisfies judgemental computation rules on points.\footnote{Without a general formal definition of higher inductive types, Lumsdaine presented this result as an informal general principle that holds for all known examples of higher inductive types.} In fact the judgemental computation rule precisely gives us commutativity of the upper triangle for the diagonal filler. As a special case of this we can consider the truncation map \(|-|\) for propositional truncation. We will use this idea to recover the original version of Kraus' paradox.

  If we augment type theory with the mapping cylinder higher inductive type then we can additionally show the following.
  \begin{enumerate}
  \item Cofibrations, homotopy equivalences and fibrations give a model structure on the category of contexts in type theory.\footnote{Lumsdaine refers to this as a pre-model-structure, since the category of contexts is not complete or cocomplete, and the choice of factorisation is not functorial, and these conditions are often assumed as part of the definition of model category.}
  \item Every cofibration is a retract of a point constructor of a higher inductive type.
  \end{enumerate}
\end{exa}

\noindent We recall the following basic proposition for ``realigning'' maps using cofibrations.
\begin{prop}
  \label{prop:realignwithcof}
  Suppose we are given a cofibration \(m : A \to B\) and maps \(f : A \to C\) and \(g : B \to C\) such that \(g \circ m \sim f\)
  \[
    \begin{tikzcd}
      A \rar["f"] \dar["m", swap] \ar[dr, phantom, "\sim" pos=0.2] & C \\
      B \ar[ur, swap, "g"] & {}
    \end{tikzcd}
  \]
  Then we can find \(g' : B \to C\) such that \(g' \circ m = f\).
\end{prop}

\begin{proof}
  Let \(h : A \to P C\) be the map witnessing the homotopy relation. In particular \(p_{0} \circ h = g \circ m\) and \(p_{1} \circ h = f\). Note that we have the commutative square given by the solid lines in the diagram below.
  \[
    \begin{tikzcd}
      A \dar["m"] \rar["h"] & P C \dar["p_{0}"] \\
      B \rar["g"] \urar[dotted, "j" description] & C
    \end{tikzcd}
  \]
  Note that \(p_{0}\) is a fibration, since it is the composition of the fibration \(p : P C \to C \times C\) and the projection map \(C \times C \to C\), which is a pullback of the fibration \(C \to 1\) and so a fibration. Furthermore, \(p_{0}\) is a weak equivalence by 3-for-2, and so a trivial fibration.
  Hence the square has a diagonal filler, \(j\), as in the dotted arrow in the diagram. We can then define \(g' := p_{1} \circ j\), giving \(g' \circ m = p_{1} \circ j \circ m = p_{1} \circ h = f\).
\end{proof}

\noindent
Note that by construction we furthermore have \(g' \sim g\) in the above proposition.

\section{hPropositions and propositional truncation}
\label{sec:hprop-prop-trunc}

In many cases higher inductive types can be naturally formulated in path categories. We give such a formulation for propositional truncation here. Our definition is based on the type theoretic definition of propositional truncation with a judgemental computation rule for the point constructor as in \cite[Section 8]{kectanonex}, since that is most relevant for our results here on Kraus' paradox. We point out that other, less strict versions of propositional truncation are also possible.

\begin{defi}
  A fibration \(f : A \to B\) in a path category is an \emph{hProposition} if the diagonal map \(A \to A \times_{B} A\) is a weak equivalence.

  A \emph{propositional truncation} of an arbitrary fibration \(f : A \to B\) is a factorisation of \(f\) as \(f = f' \circ i\) where \(f'\) is an hProposition and \(i\) has the left lifting property against all hPropositions. We will refer to \(i\) as the \emph{truncation map}.
\end{defi}

\begin{prop}
  If a map \(i : A \to C\) has the left lifting property against all hPropositions, then it is a cofibration.
\end{prop}

\begin{proof}
  Note that it suffices to show that every trivial fibration \(g : D \to E\) is an hProposition. However, this is easy to check: since trivial fibrations are closed under pullback and composition, the map \(D \times_{E} D \to E\) is a trivial fibration, and so the diagonal \(D \to D \times_{E} D\) is a weak equivalence by 3-for-2.
\end{proof}

\begin{exa}
  In any extensional path category, a map \(f : A \to B\) is an hProposition if and only if it is a monomorphism. Propositional truncations exist and are stable under pullback iff the underlying category is regular \cite{awodeybauerpat, maiettimodular}. We note that the truncation map \(A \to C\) can only be a monomorphism when \(f\) is already monic, and in this case the truncation map is an isomorphism.
\end{exa}

\begin{exa}
  We can explicitly describe the hPropositions in groupoids as follows. Note that for a fibration \(f\), the diagonal map \(A \to A \times_{B} A\) is an equivalence iff either projection map \(A \times_{B} A \to A\) is a trivial fibration. This is the case iff for all objects \(a\) of \(A\), the fibre of the projection \(A \times_{B} A \to A\) is a contractible category, which in groupoids precisely says that \(f\) is full and faithful.

  A propositional truncation of \(f\) is a factorisation of \(f\) as a bijective on objects functor followed by a full and faithful functor, which exists by taking the groupoid with the same objects as \(A\), and defining a morphism \(a \rightsquigarrow a'\) to be a morphism \(f(a) \rightsquigarrow f(a')\) in \(B\).
\end{exa}

\newcommand{\istype}{\;\operatorname{Type}}

\begin{exa}
  In the syntactic category of type theory, as in Example~\ref{ex:ttfibcat}, we can use propositional truncation in type theory to construct propositional truncations in the category of contexts. For this to work, we need to formulate propositional truncation using an induction principle with a judgemental computation rule for induction as in \cite[Section 8]{kectanonex}. To verify that truncation maps \((\Gamma,A) \to (\Gamma,\|A\|)\) do have the left lifting property against all hPropositions, we first consider the special case where an hProposition is a projection of the form \((\Gamma,z : \|A\|,B) \to (\Gamma,\| A \|)\) where \(\Gamma, z : \| A \| \vdash B \istype\), and the lower map in the lifting problem is the identity. Note that the induction term precisely gives a diagonal filler, with the upper triangle commuting using the judgemental computation rule.
  \[
    \begin{tikzcd}
      (\Gamma,A) \rar["f"] \dar[swap, "|-|"] & (\Gamma,z : \|A\|,B) \dar \\
      (\Gamma,\| A \|) \rar[equal] \ar[ur, dotted, "{\mathsf{ind}(B, f)}" description] & (\Gamma, \| A \|)
    \end{tikzcd}
  \]
  Next, note that for arbitrary lifting problems against maps of the form \((\Delta, B) \to \Delta\) we can pull back along the bottom map to obtain a lifting problem of the form above, as illustrated below.
  \[
    \begin{tikzcd}
      (\Gamma, A) \rar \dar & (\Gamma, z : \| A \|, B[\sigma]) \rar \pbcorner \dar & (\Delta, B) \dar \\
      (\Gamma, \| A \|) \rar[equal] \ar[ur, dotted] & (\Gamma, \|A\|) \rar["\sigma"] & \Delta
    \end{tikzcd}
  \]

  Finally, note that every hProposition is a fibration, and so in particular a retract of a dependent projection by \cite[Lemma 11]{gambinogarner}. Since maps with the right lifting property against a given map are closed under retracts, we can deduce that truncation maps have the left lifting property against all hPropositions.

  It is clear that the projection map \((\Gamma, \| A \|) \to \Gamma\) is an hProposition, and so we are done.
\end{exa}

\section{Univalence in path categories}
\label{sec:univ-fibr-categ}

\begin{defi}
  For a fixed fibration \(\Udot \to \U\), we say a fibration \(f : A \to B\) is \emph{\(\U\)-small} if it is a pullback of \(\Udot \to \U\) along some map \(B \to \U\). We say an object \(A\) is \emph{\(\U\)-small} if the unique map \(A \to 1\) is \(\U\)-small.
\end{defi}

\noindent
We will use the definition of univalence in path categories due to Van den Berg \cite[Definition 2.13]{vdbergunivalentpoly}. For completeness, we spell out below the details of the definition and his proof that in type theory it is equivalent to the more usual one.

We first recall the following fact.
\begin{lem}
  \label{lem:coherentpathobj}
  Given any choice of path object \(P \U\) on \(\U\), we can choose a path object \(P \Udot\) on \(\Udot\) and fibration \(P \Udot \to P \U\) such that the following diagram commutes, and the canonical map \(P \Udot \,\to\, P \U \times_{\U \times \U} \Udot \times \Udot\) given by the right hand commutative square is a fibration.
  \[
    \begin{tikzcd}
      \Udot \rar \dar & P \Udot \dar \rar & \Udot \times \Udot \dar \\
      \U \rar & P \U \rar & \U \times \U
    \end{tikzcd}
  \]

  Furthermore, if we are given two such path objects \(P \Udot\) and \(P \Udot'\), then there is a weak equivalence \(P \Udot \to P \Udot'\) making the following triangles commute.
  \[
    \begin{gathered}
      \begin{tikzcd}[column sep = 1.2em]
        P \Udot \ar[rr] \drar & & P \Udot' \dlar \\
        & P \U &
      \end{tikzcd}
    \end{gathered}
    \qquad
    \begin{gathered}
      \begin{tikzcd}[column sep = 1.2em]
        P \Udot \ar[rr] \ar[dr] & & P \Udot' \ar[dl] \\
        & \Udot \times \Udot &
      \end{tikzcd}
    \end{gathered}
  \]
\end{lem}

\begin{proof}
  We first check that \(P \U \times_{\U \times \U} \Udot \times \Udot\) is a well defined object. Note that \(\Udot \times \Udot \to \U \times \U\) factors as \(\Udot \times \Udot \to \Udot \times \U \to \U \times \U\) where both maps are pullbacks of \(\Udot \to \U\), and so fibrations. Hence \(\Udot \times \Udot \to \U \times \U\) is a fibration, and so its pullback along \(P\U \to \U \times \U\) exists.
  
  We can factor the canonical map \(\Udot \to P \U \times_{\U \times \U} \Udot \times \Udot\) as a weak equivalence followed by a fibration by \cite[Proposition 2.3]{vdbergmoerdijkpathcat}, and this precisely gives us the required structure by composing with the pullback projection maps.

  For the second part we apply \cite[Lemma 2.25]{vdbergmoerdijkpathcat} to the factorisation above and again compose with the pullback projection maps. To see that the resulting map \(P \Udot \to P \Udot'\) is a weak equivalence, we use that weak equivalences are closed under homotopy, see the second paragraph of the proof of \cite[Theorem 2.16]{vdbergmoerdijkpathcat}.
\end{proof}

\noindent
See also \cite[Theorem 2.28]{vdbergmoerdijkpathcat} for a related construction.

\begin{defi}[Van den Berg]
  \label{def:univalence}
  A fibration \(\Udot \to \U\) is \emph{univalent} if it satisfies the following. Suppose that we are given \(\U\)-small maps \(A \to C\) and \(B \to C\), witnessed by pullback squares as explicitly labelled below, together with a weak equivalence \(e : A \to B\) over \(C\).
  \begin{equation}
    \label{eq:univalenceconditions}
      \begin{gathered}
      \begin{tikzcd}
        A \rar["i"] \dar \pbcorner & \Udot \dar \\
        C \rar["f"] & \U
      \end{tikzcd}
    \end{gathered}
    \quad
    \begin{gathered}
      \begin{tikzcd}
        B \rar["j"] \dar \pbcorner & \Udot \dar \\
        C \rar["g"] & \U
      \end{tikzcd}
    \end{gathered}
    \qquad
    \begin{gathered}
      \begin{tikzcd}
        A \ar[rr, "e"] \drar & & B \dlar \\
        & C &
      \end{tikzcd}
    \end{gathered}
  \end{equation}

  Then we have homotopies \(f \sim g\) and \(j \circ e \sim i\), which are coherent, in the sense that if \(P\Udot\) is chosen as in Lemma~\ref{lem:coherentpathobj}, the homotopies can be given by maps \(C \to P \U\) and \(A \to P \Udot\) making the square on the right below commute.
  \[
    \begin{gathered}
      \begin{tikzcd}
        A \ar[dr, "i"] \ar[dd, swap, "e"] \ar[ddd, bend right] & \\
        \ar[r, phantom, "\sim"] & \Udot \ar[dd] \\
        B \ar[ur, swap, "j"] \ar[d] & \\
        C \ar[r, bend left, "f"] \ar[r, phantom, "\sim"] \ar[r, bend right, swap, "g"] & \U
      \end{tikzcd}
    \end{gathered}
    \qquad
    \begin{gathered}
      \begin{tikzcd}
        A \rar \dar & P \Udot \dar \\
        C \rar & P \U
      \end{tikzcd}
    \end{gathered}
  \]
\end{defi}

\begin{prop}
  The definition of univalence is independent of the particular choice of path objects and map \(P \Udot \to P \U\).
\end{prop}

\begin{proof}
  First note that the choice of path object \(P \U\) makes no difference by the argument below \cite[Corollary 2.10]{vdbergmoerdijkpathcat}. Now note that given a homotopy for the upper triangle given relative to a path object \(P \Udot'\), we can use the weak equivalence \(P \Udot \to P \Udot'\) given in the second part of Lemma~\ref{lem:coherentpathobj} to construct a homotopy relative to \(P \Udot\) which is still coherent.
\end{proof}

\begin{thm}[Van den Berg]
  \label{thm:univalencetounivalence}
  Suppose that \(\CC\) is a path category constructed from a model of type theory, as in Example~\ref{ex:ttmodelfibcat}, and that the model satisfies function extensionality. Then a fibration \(\Udot \to \U\) satisfies univalence in the sense of Definition~\ref{def:univalence} iff there is a term witnessing that it is univalent according to the usual formulation in type theory.
\end{thm}

\begin{proof}
  Path categories arising from directly from models of type theory satisfy the additional conditions of \emph{type theoretic fibration categories}, as formulated by Shulman in \cite{shulmanuidhc}. In particular, this allows us to define the universal equivalence between pairs of types, as in \cite[Section 5]{shulmanuidhc}. This can be used to replace the quantification over all objects \(C\) in the definition of univalence with one particular \(C\). Explicitly, in type theoretic notation, we take \(C := \sum_{A, B : \U} A \simeq B\) with \(A\), \(B\) and \(e\) given by the appropriate projections.

  Since the definition of univalence is independent of the choice of path object, we can make an explicit choice of path objects that is most convenient for the proof. We define both path objects using type theoretic notation, which can then be interpreted into the type theoretic fibration category, following Shulman. We take \(P \U := \sum_{A, B : \U} A = B\) and we define \(P\Udot\) as below.
  \[
    P \Udot := \sum_{A, B : \U} \sum_{p : A = B} \sum_{a : A} \sum_{b : B} p_{\ast}(a) = b
  \]

  We can now see that the homotopies precisely give us a term giving for all \(A, B : \U\) and all \(e : A \simeq B\) a choice of equality \(p : A = B\) together with an equality \(q(a) : p^{\ast}(a) = e(a)\) for each \(a : A\). Note that this precisely says that for every such \(e\), we have a path \(p : A = B\) such that \(e\) is homotopic to transport along \(p\). Assuming function extensionality this says that we can find \(p\) such that \(e\) is equal to transport along \(p\), i.e.\ that \(\mathsf{idtoeqv}\) has a section. It is clear that this follows from type theoretic univalence, which states that \(\mathsf{idtoeqv}\) is an equivalence. However, the converse can also be proved in type theory, as shown by Orton and Pitts \cite[Theorem 3.5]{ortonpittsdecompose},\footnote{This can be seen as an instance of very general constructions in type theory. See e.g.\ \cite[Exercise 11.8(d)]{rijkebook} and \cite[Section 3.27]{escardoufnotes}.} which can then be interpreted into the model.
\end{proof}

\begin{rem}
  It is an open problem to determine whether or not this formulation of univalence implies function extensionality. If it does, then the assumption of function extensionality can be dropped from Theorem~\ref{thm:univalencetounivalence}.
\end{rem}

\begin{defi}
  We say a path category is \emph{univalent} if for every fibration \(f : A \to B\) there is a univalent fibration \(\Udot \to \U\) for which \(f\) is \(\U\)-small:
  \[
    \begin{tikzcd}
      A \dar["f"] \rar \pbcorner & \Udot \dar \\
      B \rar & \U
    \end{tikzcd}
  \]
\end{defi}

\noindent
We will make heavy use of the following basic fact.
\begin{prop}
  Let \(\Udot \to \U\) be any fibration. Every \(\U\)-small object \(A\) is a subobject of \(\Udot\), witnessed by a monomorphism \(\iota_{A} : A \to \Udot\).
\end{prop}

\begin{proof}
  By the definition of \(\U\)-small object, we have a pullback diagram as below.
  \[
    \begin{tikzcd}
      A \rar \dar \pbcorner & \Udot \dar \\
      1 \rar & \U
    \end{tikzcd}
  \]
  The bottom map is a monomorphism, since its domain is the terminal object. The top map is then a pullback of a monomorphism, and so also a monomorphism. We take \(\iota_{A}\) to be the top map.
\end{proof}

\begin{rem}
  Although we think of \(\Udot \to \U\) as acting like a universe that contains every type belonging to the fibrations \(A \to B\), this isn't really accurate. We only required univalence, and no other closure conditions, such as \(\Sigma\) or \(\Pi\) types or that the universe contains any other types such as \(0\) and \(1\).
  
  In particular, if \(\CC\) is a preorder then every fibration is univalent since we only require that certain diagrams either commute or commute up to homotopy, and in preorders every diagram commutes. Hence every preorder path category is univalent. By applying this to the extensional path category on a lattice, we can see that univalence is not strong enough to show that coproducts are disjoint.
\end{rem}

\begin{prop}
  A path category is both extensional and univalent if and only if it is a bounded \(\wedge\)-semi lattice.
\end{prop}

\begin{proof}
  As noted in the above remark, to satisfy univalence we only require that certain diagrams commute or commute up to homotopy, and so it is vacuously true for preorders. A preorder has all finite limits iff it is a bounded \(\wedge\)-semi lattice. Therefore we confirm that a bounded \(\wedge\)-semi lattice with any path category structure is univalent, in particular the extensional path category.

  For the converse, assume that a path category \(\CC\) is both univalent and extensional. For any \(B \in \CC\), we consider the automorphism \(\tau : B \times B \to B \times B\) that swaps the two components of the product. In extensional path categories, two maps are homotopic iff they are equal. Hence univalence gives us a universe \(\Udot \to \U\) containing \(B \times B\) such that \(\iota_{B \times B} \circ \tau = \iota_{B \times B}\). Since \(\iota_{B \times B}\) is a monomorphism, we deduce that \(\tau = 1_{B \times B}\). However, we can now deduce that any two maps \(A \rightrightarrows B\) are equal.
\end{proof}

\begin{rem}
  The \emph{equivalence extension property} \cite{sattlermodelstructures, awodeyccmc} is another categorical formulation of univalence. It follows from the existence of a univalent universe, and is a useful concept to consider when constructing a univalent universe. However, viewed strictly as a property of a path category it is much weaker than univalence. In particular every extensional path category satisfies the equivalence extension property.
\end{rem}

\noindent
For this paper we don't require the full univalence condition, but only the following weaker version.
\begin{prop}
  \label{prop:weakunivalence}
  Suppose \(\Udot \to \U\) is univalent. If we are given two \(\U\)-small fibrations \(f : A \to C\) and \(g : B \to C\), witnessed by pullback squares with top maps \(i\)~and~\(j\) and a weak equivalence \(e\) from \(A\) to \(B\) over \(C\), as in \eqref{eq:univalenceconditions} in the definition of univalence, then \(j \circ e \sim i\).
\end{prop}

\begin{proof}
  This is one of the homotopies in the definition of univalence, so we simply forget the other homotopy.
\end{proof}

\section{Kraus' paradox}
\label{sec:kraus-paradox}

We now have enough general theory to give the categorical version of Kraus' paradox. We first formulate the definition of transitive type in path categories as follows.
\begin{defi}
  For a fibration \(\Udot \to \U\), a \(\U\)-small object \(A\) is \emph{\(\U\)-homogeneous} if the maps \(\iota_{A} \circ \pi_{i} : A \times A \to \Udot\) are homotopic, as illustrated below.
  \[
    \begin{tikzcd}
      & A \drar["\iota_{A}"] & \\
      A \times A \urar["\pi_{0}"] \drar[swap, "\pi_{1}"] & \sim & \Udot \\
      & A \urar[swap, "\iota_{A}"] &
    \end{tikzcd}
  \]
\end{defi}

\noindent
We now give the main theorem, corresponding to \cite[Theorem 8.12]{kectanonex}.
\begin{thm}
  \label{thm:main}
  Suppose that a \(\U\)-small object \(A\) is \(\U\)-homogeneous. Let \(m : A \to B\) be a cofibration and \(s : B \to A\) be any map. Then \(m\) is a monomorphism.
\end{thm}

\begin{proof}
  By \(\U\)-homogeneity of \(A\), we have \(\iota_{A} \circ \pi_{0} \sim \iota_{A} \circ \pi_{1}\), and so by Theorem~\ref{thm:calcfrac}~(\ref{it:calcfrac:intr}), \(\iota_{A} \circ \pi_{0} \circ \langle 1_{A}, s \circ m \rangle \sim \iota_{A} \circ \pi_{1} \circ \langle 1_{A}, s \circ m \rangle\). The left hand side is equal to \(\iota_{A}\) whereas the right hand side is equal to \(\iota_{A} \circ s \circ m\) and so \(\iota_{A} \sim \iota_{A} \circ s \circ m\).

  We illustrate this with the diagram below.
  \[
    \begin{tikzcd}
      & & A \drar["\iota_{A}"] & \\
      A \ar[urr, bend left, equal] \ar[dr, "m"] \ar[r, dotted] & A \times A \ar[ur, "\pi_{0}"] \ar[dr, "\pi_{1}"] & \sim & \Udot \\
      & B \rar[swap, "s"] & A \ar[ur, "\iota_{A}"] &
    \end{tikzcd}
  \]
  However, we now note that we have the following homotopy commutative triangle.
  \[
    \begin{tikzcd}[sep = 1.5em]
      A \ar[rr, "\iota_{A}"] \ar[dd, "m"] \ar[dr, phantom, "\sim"] & & \Udot \\
      & A \urar[swap, "\iota_{A}"] & \\
      B \urar[swap, "s"] & & {}
    \end{tikzcd}
  \]
  Hence by Proposition~\ref{prop:realignwithcof} we can find \(j : B \to \Udot\) such that \(j \circ m\) is strictly equal to the monomorphism \(\iota_{A}\). Hence \(m\) must also be a monomorphism.
\end{proof}

\noindent
So far we have not used univalence. However, we do need it in order to get non trivial examples of \(\U\)-homogeneous types. We show how to obtain \(\U\)-homogeneous types from the more natural notion of homogeneity below.
\begin{defi}
  \label{defn:homogeneous}
  A type \(A\) is \emph{homogeneous} if there is a weak equivalence \(e\) from \(A \times A \times A\) to \(A \times A \times A\) over \(A \times A\) such that for the two sections \(s_{0} := \langle \pi_{0}, \pi_{1}, \pi_{0} \rangle\) and \(s_{1} := \langle \pi_{0}, \pi_{1}, \pi_{1} \rangle\), we have \(\pi_{2} \circ e \circ s_{0} \sim \pi_{2} \circ s_{1}\).
  \[
    \begin{tikzcd}
      A \times A \times A \ar[rr, "e"] \ar[dr, "{\langle \pi_{0}, \pi_{1} \rangle}" description] & & A \times A \times A \ar[dl, "{\langle \pi_{0}, \pi_{1} \rangle}" description] \\
      & A \times A \ar[ul, bend left, "s_{0}"] \ar[ur, bend right, swap, "s_{1}"] &
    \end{tikzcd}
  \]
\end{defi}

\begin{lem}
  \label{lem:homogentouhomogen}
  Suppose that \(A\) is \(\U\)-small and homogeneous and that \(\U\) is univalent. Then \(A\) is \(\U\)-homogeneous.
\end{lem}

\begin{proof}
  We note that we have the following pullback diagram.
  \[
    \begin{tikzcd}
      A \times A \times A \rar["\pi_{2}"] \dar \pbcorner & A \rar["\iota_{A}"] \dar \pbcorner & \Udot \dar \\
      A \times A \rar & 1 \rar & \U
    \end{tikzcd}
  \]
  Hence applying Proposition~\ref{prop:weakunivalence} with the equivalence \(e\) given by homogeneity, we have \(\iota_{A} \circ \pi_{2} \circ e \sim \iota_{A} \circ \pi_{2}\), and so \(\iota_{A} \circ \pi_{2} \circ e \circ s_{0} \sim \iota_{A} \circ \pi_{2} \circ s_{0} = \iota_{A} \circ \pi_{0}\). From the remaining part of the definition of homogeneity, we have \(\pi_{2} \circ e \circ s_{0} \sim \pi_{2} \circ s_{1}\), and so \(\iota_{A} \circ \pi_{2} \circ e \circ s_{0} \sim \iota_{A} \circ \pi_{2} \circ s_{1}\). Combining these two homotopies using Theorem~\ref{thm:calcfrac}~(\ref{it:calcfrac:eqreln}) together with the equations \(\pi_{2} \circ s_{0} = \pi_{0}\) and \(\pi_{2} \circ s_{1} = \pi_{1}\) gives us \(\iota_{A} \circ \pi_{1} \sim \iota_{A} \circ \pi_{0}\).

  We can illustrate this with the diagram below.\footnote{Although the homotopy \(e \circ s_{0} \sim s_{1}\) in the diagram is not part of the definition of homogeneity, it is also true. To see this, note that in general, for objects \(X\) and \(Y\), we can show using \cite[Proposition 2.7]{vdbergmoerdijkpathcat} (``Brown's lemma'') that \(PX \times PY\) is a path object for \(X \times Y\). Hence we can define homotopies into products componentwise. In this case, we combine the homotopy from Definition~\ref{defn:homogeneous} with the trivial homotopy. We can see from the algebraic proof that this is not strictly needed for the lemma.}
  \[
    \begin{tikzcd}[sep=1.2em]
      & A \times A \times A \ar[r, "\pi_{2}"] \ar[dd, "e"] & A \ar[dr, "\iota_{A}"] & \\
      A \times A \ar[r, phantom, "\sim"] \ar[ur, "s_{0}" description] \ar[dr, "s_{1}" description] \ar[urr, bend left = 5em, "\pi_{0}"] \ar[drr, bend right = 5em, swap, "\pi_{1}"] & \ar[rr, phantom, "\sim" description] & & \Udot \\
      & A \times A \times A \ar[r, "\pi_{2}"] & A \ar[ur, swap, "\iota_{A}"] &
    \end{tikzcd}
    \qedhere
  \]
\end{proof}

\begin{cor}
  \label{cor:usmallwellsupp}
  If \(A\) is homogeneous and \(\U\)-small for some univalent fibration \(\Udot \to \U\), \(m : A \to B\) is a cofibration and there is some map \(B \to A\), then \(m\) is a monomorphism.
\end{cor}

\begin{proof}
  We apply Lemma~\ref{lem:homogentouhomogen} and Theorem~\ref{thm:main}.
\end{proof}

\begin{cor}
  \label{cor:homgwellsup}
  If \(\CC\) is univalent, \(m : A \to B\) is a cofibration, \(A\) is homogeneous and there is any map at all \(s : B \to A\), then \(m\) is a monomorphism.
\end{cor}

\begin{proof}
  Since \(\CC\) is univalent, \(A\) is \(\U\)-small for some univalent \(\U\). Hence we can apply Corollary~\ref{cor:usmallwellsupp}.
\end{proof}

\begin{cor}
  \label{cor:pointed}
  If \(\CC\) is univalent and \(A\) is an object that is homogeneous and has a point \(a_{0} : 1 \to A\), then any cofibration \(A \to B\) is a monomorphism.
\end{cor}

\begin{proof}
  We define \(s : B \to A\) to be \(a_{0} \circ \mathord{!}_{B}\) and apply Corollary~\ref{cor:homgwellsup}.
\end{proof}

\begin{cor}
  Suppose that \(\CC\) is univalent and that whenever \(m : A \to B\) is a cofibration, so is \(C \times m : C \times A \to C \times B\) for any object \(C\). Then any cofibration with homogeneous domain is a monomorphism.
\end{cor}

\begin{proof}
  Suppose we are given two maps \(f, g : C \to A\) such that \(m \circ f = m \circ g\). We take the full subcategory of \(\CC / C\) consisting of fibrations. We recall that this is a path category \cite[Definition 2.5]{vdbergmoerdijkpathcat} and that \(C \times \Udot \to C \times \U\) is univalent. Also note that \(C \times A \to C\) is pointed, e.g.\ by \(\langle 1_{C}, f \rangle\). Hence we can apply Corollary~\ref{cor:pointed}.
\end{proof}

\section{Some worked examples in groupoids}
\label{sec:work-example-group}

For this section we work in the category of groupoids as in Example~\ref{ex:groupoids}, and we will see a few simple examples to illustrate the interaction between homogeneous types, cofibrations and univalence. We point out that some of the ideas that appear here can be thought of as simplified special cases of higher group theory, in particular the approach developed in \cite{buchholtzvdoornrijke} using HoTT. In particular, Theorem~\ref{thm:abeliannotunivalent} makes use of the fact that abelian groups are homogeneous, roughly corresponding to the link between central types and homogeneous types appearing in \cite{bctrcentral}. Also recall from \cite[Section 4.4]{buchholtzvdoornrijke} that for a \(1\)-group \(G\) in HoTT, elements of the centre of \(G\) correspond precisely to loops in the type of automorphisms of \(BG\).

For the first example, note that using the explicit description of cofibration as maps that are injective on objects, it is easy to see that every cofibration with discrete domain is monic. However, assuming the law of excluded middle, we can also view it as a special case of Kraus' paradox. Using excluded middle, we can show that every discrete groupoid is homogeneous. This allows us to apply Corollary~\ref{cor:usmallwellsupp} together with the fact that any discrete fibration is \(\U\)-small for a univalent universe \(\U\) \cite[Section 5.4]{hofmannstreichergroupoid}, to give another proof that such cofibrations must be monic.

For the remaining examples, we recall that any group \(G\) can be viewed as a one object groupoid \(BG\), and that this defines an embedding from groups to groupoids. We write the object of \(BG\) as \(\ast\).

\begin{prop}
  \label{prop:groupnomonictrunc}
  The truncation map \(BG \to \| BG \|\) is monic iff \(G\) is trivial.
\end{prop}

\begin{proof}
  In an hProposition, any two paths between two objects must be equal. In particular for \(g \in G\), we must have \(|g| = 1_{|\ast|}\), and so if \(|-|\) is monic, we can deduce \(g = 1_{\ast}\).
\end{proof}

\begin{thm}
  \label{thm:abeliannotunivalent}
  If \(G\) is non trivial and abelian, then \(BG\) is not \(\U\)-small for any univalent fibration \(\U\).
\end{thm}

\begin{proof}
  By Proposition~\ref{prop:groupnomonictrunc} and the assumption \(G\) is non trivial, the truncation map \(B G \to \| B G \|\) cannot be monic. Hence by Corollary~\ref{cor:usmallwellsupp}, to show \(BG\) is not \(\U\)-small for univalent \(\U\) it suffices to show that \(BG\) is homogeneous. Using the embedding from groups to groupoids, we can instead show the corresponding property of the group~\(G\). That is, we need an automorphism \(\theta : G \times G \times G \to G \times G \times G\) such that \(\pi_{2} \circ \theta \circ s_{0} \sim \pi_{2} \circ \theta \circ s_{1}\). We will in fact ensure that \(\theta \circ s_{0} = \theta \circ s_{1}\). We define \(\theta\) as follows.
  \[
    \theta(g, h, k) := (g, h, h g^{-1} k)
  \]
  We check that this is a group isomorphism. It is clear that \(\theta\) preserves the identity. Since \(G\) is abelian, we can see that \(\theta\) preserves group multiplication, and so is a homomorphism. It is also clear that it has an inverse defined similarly, and so is an isomorphism. However, we have now checked that \(BG\) is homogeneous and so we can apply Corollary~\ref{cor:usmallwellsupp}.
\end{proof}

\noindent
On the other hand, there are examples of non abelian \(G\) that are pullbacks of a univalent fibration. In particular we have the following.
\begin{thm}
  Suppose that \(G\) is a \emph{complete} group. That is, \(G\) has trivial centre and every automorphism of \(G\) is an inner isomorphism (i.e.\ of the form \(\lambda x.h^{-1}xh\) for some \(h\)). Then \(BG \to 1\) is univalent.
\end{thm}

\begin{proof}
  First note that if \(B \to A\) is a pullback of \(BG \to 1\), then we have a canonical isomorphism \(B \cong A \times BG\), and that equivalences over \(A\) are precisely automorphisms of \(A \times BG\) over \(A\). Each automorphism of \(A \times BG\) over \(A\) is determined by a homomorphism \(\theta : A \times BG \to BG\).

  Note that for every object \(a\) of \(A\), the map \(\theta\) restricts to an automorphism \(\theta(1_{a}, -)\) of \(G\). By assumption, this must be an inner automorphism. That is, for all objects \(a\) of~\(A\) there must exist \(h_{a}\) such that \(\theta(1_{a}, g) = h_{a}^{-1}gh_{a}\).

  Now for any objects \(a, a'\) of \(A\), any path \(p : a \rightsquigarrow a'\) and any element \(g\) of \(G\) we have the following.
  \begin{align*}
    h_{a'}\theta(p, 1_{\ast})h_{a}^{-1} g &= h_{a'}\theta(p, 1_{\ast})h_{a}^{-1}h_{a}\theta(1_{a}, g)h_{a}^{-1} \\
                          &= h_{a'}\theta(p, g)h_{a}^{-1} \\
                          &= h_{a'}\theta(1_{a'}, g)h_{a'}^{-1}h_{a'}\theta(p, 1_{\ast})h_{a}^{-1} \\
                          &= g h_{a'}\theta(p, 1_{\ast})h_{a}^{-1}
  \end{align*}
  We deduce that \(h_{a'}\theta(p, 1_{\ast})h_{a}^{-1}\) belongs to the centre of \(G\), and so we deduce \(h_{a'}\theta(p, 1_{\ast})h_{a}^{-1} = 1_{G}\), since \(G\) has trivial centre. Therefore \(h_{a'}\theta(p, 1_{\ast}) = h_{a}\).

  We now verify that \(h\) is a natural transformation from \(\theta : A \times BG \to BG\) to \(\pi_{1} : A \times BG \to BG\). Every morphism in \(A \times BG\) is of the form \((p, g)\) where \(p : a \rightsquigarrow a'\) in \(A\) and \(g \in G\). We then calculate as follows.
  \begin{align*}
    h_{a'} \theta(p, g) &= h_{a'}\theta(p, 1_{\ast}) \theta(1_{a}, g) \\
                        &= h_{a} \theta(1_{a}, g) \\
    &= g h_{a} \qedhere
  \end{align*}
  We can thus use \(h\) as the homotopy \(j \circ e \sim i\) required in Definition~\ref{def:univalence}. The required homotopy \(f \sim g\) is trivial.
\end{proof}

\begin{exa}
  The symmetric group \(S_{3}\) is univalent as an object in groupoids.
\end{exa}

\begin{rem}
  The above proof is partially based on an example due to Mike Shulman\footnote{This example was suggested as part of an online discussion, which is available at \url{https://hott.zulipchat.com/\#narrow/channel/228519-general/topic/Inconsistency.20of.20HITs.20that.20land.20in.20any.20universe/with/422268465}} of a ``self referential'' univalent fibration in homotopy type theory, i.e.\ a family of types \(b : B \vdash A : \U\) such that the projection \(\sum_{b : B}A \to B\) is univalent in the sense of Definition~\ref{def:univalence} and such that \(B \simeq A(b_{0})\) for some \(b_{0} : B\). He observed one can show internally in HoTT that the classifying space \(B G\) has trivial automorphism group when \(G\) is a group with trivial centre and trivial outer automorphism group, and used this to construct the example as follows. The group \(S_{n}\) satisfies the conditions for \(n \neq 2, 6\). We take \(B\) to be the coproduct of \(n\) copies of \(B S_{n}\). This is a sum of higher inductive types, and so a higher inductive type. Hence we can define \(A\) by higher recursion. We take \(A\) to be equal to \(B\) on point constructors. A path constructor consists of \(m < n\) together with an element, say \(g\) of \(S_{n}\). For the higher recursion, we need for each path constructor a proof of \(B = B\), or equivalently an equivalence \(B \simeq B\). We define the equivalence for \((m, g)\) to be \((m', z) \mapsto (g(m'), z)\).
\end{rem}

\section{Conclusion}
\label{sec:concl-direct-future}

\subsection{Variations and generalisations}
\label{sec:vari-gener}

\newcommand{\sets}{\mathbf{Set}}
\newcommand{\refl}{\mathsf{refl}}

Recently David W\"{a}rn \cite{warnkrausparadox} has noticed a variation on Kraus' paradox that works as follows. We will present it for all cofibrations, although following Kraus, the argument was originally phrased using propositional truncation.

Let \(m : A \to B\) be a cofibration with \(A\) a small type. Working internally in type theory, for each element \(a\) of \(A\) we have a contractible type \(C(a)\) defined as \(\sum_{a' : A} a' = a\), with the centre of contraction given explicitly as \((a, \refl)\). The special case of univalence for contractible types tells us that the type of all small contractible types is contractible. Hence we can lift to obtain for each \(b : B\) a type \(C(b)\) together with a proof that \(C(b)\) is contractible, as illustrated below.
\[
  \begin{tikzcd}[column sep = huge, row sep = large]
    A \ar[r] \ar[d, swap, "m"] & \sum_{C : \U} \operatorname{isContr}(C) \ar[d] \\
    B \ar[r] \ar[ur, dotted] & 1
  \end{tikzcd}
\]
In particular this includes a witness that \(C(b)\) is inhabited, say \(c(b) : C(b)\). The upper triangle in the lifting diagram tells us that for \(a : A\), \(C(m(a))\) is definitionally equal to \(\sum_{a' : A} a' = a\) and \(c(m(a))\) is definitionally equal to \((a, \refl)\). Hence if we apply the first projection we get \(\pi_{0}(c(m(a))) \equiv a\).

In some ways this argument is more general than Kraus': it only requires the special case of univalence for contractible types, and does not require \(A\) to be homogeneous or inhabited. However, in at least one way it is, unfortunately, less general. In formal treatments of type theory the projection terms \(\pi_{0}\) include type annotations that explicitly describe the types used to construct the \(\Sigma\) type. In particular, the variable \(a\) occurs in the term \(\pi_{0}\) we used above implicitly as a part of the type annotation. In W\"{a}rn's formalisation in the Agda proof assistant the type annotation appears as an implicit parameter to the projection function that is inferred automatically by Agda. This issue prevents us from adapting the argument to arbitrary path categories.

For a specific counterexample to the statement that all cofibrations are monomorphisms in the presence of semi-univalence, we can consider the extensional path category of \(\sets\), and make the following observation, based on the cardinal model of type theory due to Bauer and Winterhalter \cite[Section 8.3]{winterhalterthesis}. We say a fibration is \emph{semi-univalent} if in Definition~\ref{def:univalence} we only have the ``lower homotopy'' \(f \sim g\).\footnote{This is a reformulation of the propositional version of the \emph{isomorphism reflection principle} from \cite[Section 8.3]{winterhalterthesis}.} Note that in type theory semi-univalence implies propositional extensionality, and so also univalence for contractible types.

\begin{prop}
  Assuming the axiom of choice, the extensional path category on \(\sets\) is semi-univalent.
\end{prop}

\begin{proof}
  Given a map \(f : A \to B\), we take \(\U\) to be the quotient of \(B\) identifying elements if their fibres have the same cardinality. The quotient map has a section \(s : \U \to B\) by the axiom of choice, and pulling back \(f\) along \(s\) gives the required universe.
\end{proof}

\noindent
The author expects nonetheless that a semantic version of W\"{a}rn's argument is possible. Note that if we use the ``standard'' interpretation of \(\Sigma\)-types in set theory as sets of ordered pairs we can define the first projection ``uniformly.'' E.g., using Kuratowski ordered pairs, we always define \(\pi_{0}(z)\) to be \(\bigcup \bigcap z\) independent of which particular \(\Sigma\)-type we are considering. Moreover, for the more standard definition of universe in \(\sets\) as \(V_{\kappa}\) for inaccessible \(\kappa\), we can restrict this uniform projection to an operation on small sets. However, we leave it for future work to find a natural formulation of this property for universes in path categories, to study which examples of path categories it applies to, and to show that under those conditions univalence for contractible types implies that every cofibration is a monomorphism.

The variation of Kraus' paradox due to the author in \cite[Section 6]{SwanPathId} is another example where the requirement of homogeneous domain is dropped, but which turns out to be less general in other ways. For one thing, it requires the very limiting condition that the cofibration \(m : A \to B\) is also an embedding in type theory.

However, the current version of Kraus' paradox is likely not the most powerful. One can naturally ask the following questions, which are all, to the author's knowledge, open problems:
\begin{enumerate}
\item Can we say any more about cofibrations with homogeneous domain, other than that they are monomorphisms?
\item Can we replace the requirement of homogeneity with something weaker? What can we still say if the domain has a rich supply of automorphisms, but not quite enough to get homogeneity?
\item Is there a better formulation of Kraus' paradox for higher inductive types with recursive point constructors, such as \(W\)-types and nullification \cite{rijkeshulmanspitters}?
\end{enumerate}

\subsection{Is Kraus' paradox surprising?}

We can now give another perspective on why Kraus' paradox is a surprising result.

In path categories we have two different notions of injective map. As in any category we can think of monomorphisms as injective maps, following our intuition from many natural examples of categories. The second, arguably more correct definition is maps that are ``homotopy monomorphic,'' i.e.\ those that factor as a weak equivalence followed by an hProposition. Certainly the latter agrees with the internal statement in type theory that the map is injective when we work in models of type theory.

In extensional path categories the two definitions are equivalent, so there is no ambiguity. For non extensional path categories, we no longer expect the two classes to coincide. For models coming from homotopy theory we have the intuition that propositional truncation no longer identifies points, but instead ``adds new paths.'' From this perspective, it is not surprising that there are at least some examples of path categories where the truncation maps are always monomorphisms, even though truncations are almost never homotopy monomorphisms. In Cisinski model structures, for example, a map is a cofibration if and only if it is monic, by definition \cite[Definition 2.12]{cisinskimodelstr}.

However, just as for extensional path categories, there's no reason that our intuitions about one class of model should apply to every model. It is easy to think of the choice of definition for cofibration as an ``implementation detail'' that is somewhat independent of the logical principles holding in a model. Hence one might na\"{i}vely still expect some flexibility here, with path categories intermediate between extensional ones and ``homotopical'' ones, where univalence holds, but truncation maps do not have to be monic. Kraus' paradox tells us that at least for homogeneous types this is wrong again: we have to take truncation maps with homogeneous domain to be monomorphisms if we want univalence to hold.

\section*{Acknowledgement}
\noindent I am grateful for useful comments, suggestions and discussions with Benno van den Berg, Ulrik Buchholtz, Mart\'{i}n Escard\'{o}, Egbert Rijke and David W\"{a}rn. In particular this paper was originally motivated by considerations around the variant of Kraus' paradox due to W\"{a}rn discussed in the conclusion. I am also grateful for the useful corrections and improvements suggested by the anonymous referees.

\bibliographystyle{alphaurl}
\bibliography{mybib}{}

\end{document}